\documentclass[amscd,amssymb,11pt]{amsart}

\usepackage{amssymb}
\usepackage[all]{xy}
\usepackage{url}

\newtheorem{thm}{Theorem}[section]
\newtheorem{lem}[thm]{Lemma}
\newtheorem{cor}[thm]{Corollary}
\newtheorem{prop}[thm]{Proposition}

\newtheorem{addendum}[thm]{Addendum}

\theoremstyle{definition}

\newtheorem{defn}[thm]{Definition}

\theoremstyle{remark}

\newtheorem{rem}[thm]{Remark}

\numberwithin{equation}{section}

\newcommand{\thmref}[1]{Theorem~\ref{#1}}
\newcommand{\corref}[1]{Corollary~\ref{#1}}
\newcommand{\secref}[1]{\S\ref{#1}}

\newcommand{\propref}[1]{Proposition~\ref{#1}}
\newcommand{\lemref}[1]{Lemma~\ref{#1}}

\newcommand{\defref}[1]{Definition~\ref{#1}}

\newcommand{\Hom}{\operatorname{Hom}}
\newcommand{\Ext}{\operatorname{Ext}}

\newcommand{\IM}{\operatorname{im}}
\newcommand{\ModA}{\operatorname{Mod}\text{-}A}

\newcommand{\A}{{ A}}
\newcommand{\U}{{\mathcal  U}}
\newcommand{\M}{{\mathcal  M}}
\newcommand{\Qq}{{\mathcal Q}}
\newcommand{\QM}{{\mathcal Q \mathcal M}}
\newcommand{\QU}{{\mathcal Q \mathcal U}}

\newcommand{\Z}{{\mathbb  Z}}

\newcommand{\R}{{\mathbb  R}}

\newcommand{\Sinfty}{\Sigma^{\infty}}
\newcommand{\Oinfty}{\Omega^{\infty}}

\newcommand{\sm}{\wedge}

\newcommand{\ra}{\rightarrow}
\newcommand{\xra}{\xrightarrow}

\newcommand{\hra}{\hookrightarrow}
\newcommand{\era}{\twoheadrightarrow}

\begin{document}

\title[Dyer-Lashof operations as extensions]{Dyer-Lashof operations as extensions of Brown-Gitler Modules}

\author[Kuhn]{Nicholas J.~Kuhn}
\address{Department of Mathematics \\ University of Virginia \\ Charlottesville, VA 22904}
\email{njk4x@virginia.edu}

\date{June 25, 2023. Revised, November 6, 2024.}


\begin{abstract}  At the prime 2, let $T(n)$ be the $n$ dual of the $n$th Brown-Gitler spectrum with mod 2 homology $G(n)$.  Our previous work on the spectral sequence computing $H_*(\Oinfty X)$ arising from the Goodwillie tower of $\Sinfty \Oinfty: \text{Spectra} \ra \text{Spectra}$ led us to observe that there are extensions between various of the right $A$-modules $G(n)$ such that splicing with these gives an action of the Dyer-Lashof algebra on $\Ext_A^{*,*}(G(\star),M)$.  

We give explicit constructions of these `Dyer-Lashof operation' extensions: one construction relates them to the cofiber sequence associated to the $\Z/2$-transfer. Another relates key `squaring' Dyer-Lashof operations to the Mahowald short exact sequences.  Finally, properties of the spectra $T(n)$ allow us to geometrically realize our extensions by cofibration sequences, with implications for the Adams spectral sequences computing $[T(\star), X]$.

\end{abstract}

\maketitle

\section{Introduction and main results} \label{introduction}

The mod 2 homology of a spectrum $X$, $H_*(X)$, is a locally finite right module over the mod 2 Steenrod algebra $A$, with Steenrod operations lowering degree.  By contrast, $H_*(\Oinfty X)$, the mod 2 homology of the associated infinite loopspace $\Oinfty X$, is a more complicated object. In particular, it admits Dyer--Lashof operations, and its right $A$--module structure will be unstable: $xSq^i = 0$ whenever $2i>|x|$.

Paper \cite{kuhn mccarty} was a study of the homology spectral sequence associated to the Goodwillie tower of $\Sinfty_+ \Oinfty: \text{Spectra} \ra \text{Spectra}$.
This converges to $H_*(\Oinfty X)$, with an $E_1$--page that is an algebraic functor of $H_*(X)$: the enveloping algebra of $\mathcal R_*(H_*(X))$.  Here $\mathcal R_*(N)_{\star}$ is the free allowable module over the Dyer-Lashof algebra generated by a right $A$--module $N$: see \defref{R def} for more detail. This is a bigraded vector space with $\mathcal R_0(N) = N$, Dyer-Lashof operations
$$ Q^r: \mathcal R_s(N)_n \ra \mathcal R_{s+1}(N)_{n+r},$$
zero for $r < n$, and Steenrod operations
$$ Sq^k: \mathcal R_s(N)_n \ra \mathcal R_{s}(N)_{n-k}.$$

Part of \cite{kuhn mccarty} was the construction of an algebraic spectral sequence which in `good cases' agreed with the topological one. The starting point for our paper here, and related to understanding the $E_{\infty}$--page of the algebraic spectral sequence, is the following observation: 

\begin{prop} \label{starting point prop} If $M$ is a locally finite right $A$--module, the bigraded vector space $ \Ext_A^{*,*}(G(\star),M)$
is a natural subquotient of $ \mathcal R_*(\Sigma^{-1}M)_{\star-1}$,
as right $A$--modules equipped with Dyer-Lashof operations.  
\end{prop} 
Here $G(n)$ is the free unstable right $A$--module on a top degree class $\iota_n$ of degree $n$. There is a natural isomorphism
$$\Hom_{A}(G(n),M) \simeq M_n$$
for all unstable right $A$--modules $M$.  The $\Ext$-groups are computed in the abelian category of locally finite right $A$-modules, or, equivalently, in the abelian category of $A_*$-comodules: see \secref{categories sec}.

The proposition tells us that for all locally finite right $A$--modules $M$, $n \geq 0$, and $r \geq 0$, there exist natural Dyer-Lashof operations
$$ Q^r:  \Ext_A^{s,s}(G(n),M) \ra \Ext_{A}^{s+1,s+1}(G(n+r), M),$$
zero for $r < n-1$, and Steenrod operations
$$ Sq^k:  \Ext_A^{s,s}(G(n),M) \ra \Ext_{A}^{s,s}(G(n-k), M),$$
satisfying the usual properties.

The Steenrod operations here are easily understood: Yoneda's lemma implies that $a \in A^k$ induces a map of right $A$--modules $a \cdot: G(n) \ra G(n+k)$.

More curious are the Dyer-Lashof operations.  Fixing $n$ and $r$, the properties of $Q^r$ imply that $Q^r$ must be induced by Yoneda splice with the extension
$$ Q(n,r) = Q^r(1_{G(n)}) \in \Ext_{A}^{1,1}(G(n+r),G(n)).$$

The extensions $Q(n,r)$ can be regarded as short exact sequences
\begin{equation} \label{Qnr defn} 0 \ra \Sigma^{-1} G(n) \ra Q(n,r) \ra G(n+r) \ra 0.
\end{equation}

The first goal of this paper is to give a construction of this interesting family of short exact sequences that relates them all to the single short exact sequence
\begin{equation} \label{basic ses} 0 \ra H_*(S^{-1}) \ra H_*(P_{-1}) \ra H_*(P_0) \ra 0
\end{equation} 
obtained by applying homology to the cofiber sequence
\begin{equation} \label{basic cofib seq} S^{-1} \ra P_{-1} \ra P_0.
\end{equation}
Here $P_k$ is the Thom spectrum of $k$ copies of the canonical line bundle over $\R \mathbb P^{\infty}$.
We also give alternative descriptions of the extensions $Q(n,n-1)$ and $Q(n,n)$, relating them to Mahowald short exact sequences, with a computationally useful consequence.

Then we observe that all these algebraic results can be topologically realized.  The $A$-module $G(n)$ arises as $H_*(T(n))$, where $T(n)$ is the $n$-dual of the $n$th Brown-Gitler spectrum.  Properties of the $T(n)$ can be used to show that there exist cofibration sequences realizing our extension in mod 2 homology.

We now describe our results in more detail.

\subsection{Algebraic results related to (\ref{basic ses})}

Given a right $A$-module $M$, let $E_M$ be the short exact sequence of $A$-modules
$$ 0 \ra \Sigma^{-1}M \ra M \otimes H_*(P_{-1}) \ra M \otimes H_*(P_0) \ra 0 $$
obtained by tensoring (\ref{basic ses}) with $M$.  The extension $E_M$ will then induce a natural transformation
$$ \delta_M: \Hom_A(N,M \otimes H_*(P_0)) \ra \Ext_A^{1,1}(N,M)$$
defined by sending $\alpha: N \ra M \otimes H_*(P_0)$ to the pullback of $E_M$ along $\alpha$.

\begin{prop} \label{unstable M prop}  If $M$ is unstable, then for all $m\geq 0$,
$$\delta_M: \Hom_A(G(m),M \otimes H_*(P_0)) \ra \Ext_A^{1,1}(G(m),M)$$ will be onto. If  also $M$ has top nonzero degree $n$ and $m \ge 2n-1$, then $\delta_M$ will be an isomorphism. 
\end{prop}

\begin{cor} If $r \geq n-1$, then 
$$\delta_{G(n)}: \Hom_A(G(n+r),G(n)\otimes H_*(P_0)) \ra \Ext_A^{1,1}(G(n+r),G(n))$$ is an isomorphism.
\end{cor}

It follows that our extension (\ref{Qnr defn}), i.e. $Q^r(i_n) \in \Ext_A^{1,1}(G(n+r),G(n))$, will correspond to an $A$-module map in $\Hom_A(G(n+r), G(n) \otimes H_*(P_0))$.

Our next result is the identification of this map.   To explain this, we need to say a little bit about the modules $G(n)$.  As will be explained in \secref{background section}, $G(n)_k$ has an additive basis which is indexed on a subset of the Milnor basis of $A^{n-k}$.  Thus $G(n) \otimes H_*(P_0)$ has a preferred basis. Since $P_0 = \Sinfty \mathbb R \mathbb P^{\infty}_+$, $H_*(P_0)$ is unstable, and thus so is $G(n) \otimes H_*(P_0)$.     Now define
$$ q(n,r): G(n+r) \ra H_*(P_0) \otimes G(n)$$
to be the unique $A$--module map sending the generator of $G(n+r)$ to the sum of all the basis elements in $(H_*(P_0) \otimes G(n))_{n+r}$.

\begin{thm} \label{main thm} $\delta_{G(n)}(q(n,r)) = Q^r(i_n)$. Equivalently, there is a pullback diagram of extensions of right $A$--modules
\begin{equation*} \label{pushout diagram}
\SelectTips{cm}{}
\xymatrix{
0  \ar[r] & \Sigma^{-1} G(n)   \ar[r]& Q(n,r) \ar[d] \ar[r]& G(n+r)  \ar[d]^{q(n,r)} \ar[r] & 0 \\
0  \ar[r] & \Sigma^{-1} G(n) \ar@{=}[u]  \ar[r]& G(n) \otimes H_*(P_{-1})  \ar[r]& G(n) \otimes H_*(P_0)   \ar[r] & 0. }
\end{equation*}
\end{thm}

Curiously the proof of \propref{unstable M prop} uses a version of the more specific \thmref{main thm} in its proof.

\begin{addendum} \label{homology addendum} Except in the simple cases when $(n,r) = (0,0)$ or $(1,0)$, one can make the bottom short exact sequence in the last diagram usefully smaller by noting that $q(n,r): G(n+r) \ra G(n) \otimes H_*(P_0)$ factors as 
$$G(n+r) \xra{\tilde q(n,r)} G(n) \otimes H_*(P_1) \hra G(n) \otimes H_*(P_0).$$
\end{addendum}

\subsection{Algebraic results about $Q(n,n-1)$ and $Q(n,n)$}

Two extreme examples of our extensions can be identified in a different way.  Let $p_m: G(m+1) \ra \Sigma G(m)$ be the nonzero $A$-module map, and recall that there are Mahowald short exact sequences (dual to \cite[Prop.2.3.3]{schwartz book})
$$ 0 \ra G(n)\xra{Sq^n \cdot} G(2n) \xra{p_{2n-1}} \Sigma G(2n-1) \ra 0,$$
and isomorphisms
$$  p_{2n}: G(2n+1)\xra{\sim} \Sigma G(2n).$$

\begin{prop} \label{Q(n,n-1) prop}
(a) There is an isomorphism of extensions
\begin{equation*}
\SelectTips{cm}{}
\xymatrix{
0  \ar[r] &  \Sigma^{-1} G(n) \ar[r]^{Sq^n \cdot}& \Sigma^{-1}G(2n) \ar[d]^{\wr} \ar[r]^{p_{2n-1}}&    G(2n-1)\ar[r] & 0
 \\
0  \ar[r] &  \Sigma^{-1} G(n)\ar@{=}[u]  \ar[r]& Q(n,n-1)   \ar[r]&   G(2n-1)\ar@{=}[u] \ar[r] & 0, }
\end{equation*}
where the top row is one desuspension of a Mahowald short exact sequence. \\

(b) There is a pushout of extensions
\begin{equation*}
\SelectTips{cm}{}
\xymatrix{
0  \ar[r] &  \Sigma^{-2} G(n+1) \ar[r]^{Sq^{n+1} \cdot} \ar[d]^{p_{n}} & \Sigma^{-2}G(2n+2) \ar[r]^{p_{2n+1}} \ar[d]&  \Sigma^{-1} G(2n+1) \ar[d]^{\wr} \ar[r] & 0  \\
0  \ar[r] &  \Sigma^{-1} G(n)   \ar[r] & Q(n,n)  \ar[r]&   G(2n)  \ar[r] & 0}
\end{equation*}
where the top sequence is two desuspensions of a Mahowald short exact sequence.  Thus the bottom sequence is isomorphic to the top one if $n$ is even.
\end{prop}

Let $h_0 \in \Ext_A^{1,1}(\Z/2,\Z/2)$ denote the nonsplit extension of right $A$--modules
$$ 0 \ra \Sigma \Z/2 \ra G(2) \ra \Sigma^2 \Z/2 \ra 0.$$  ($G(2) = H_*(\mathbb R \mathbb P^2)$.) In the usual way, this induces a natural map
$$ h_0 : \Ext_A^{s,t}(N,M) \ra \Ext_A^{s+1,t+1}(N,M)$$
for all right $A$--modules $M$ and $N$.

We will use \propref{Q(n,n-1) prop}(b) to prove the following theorem.

\begin{thm} \label{squaring thm}  For all right $A$--modules $M$, the following diagram commutes
\begin{equation*}
\SelectTips{cm}{}
\xymatrix{
\Ext_A^{s,s}(G(2n),M) \ar[d]^{Sq^n} \ar[rrr]^-{h_0 } &&& \Ext_A^{s+1,s+1}(G(2n),M) \ar[d]^{Sq^n}  \\
\Ext_A^{s,s}(G(n),M) \ar[urrr]^-{Q^n} \ar[rrr]^-{(n+1)h_0} &&& \Ext_A^{s+1,s+1}(G(n),M). }
\end{equation*}
\end{thm}

\subsection{Topological realization}

The (co)fibration sequence (\ref{basic cofib seq}) extends to
$$ S^{-1} \ra P_{-1} \ra P_0 \xra{tr} S^0,$$
where $tr$ is the $\Z/2$--transfer. Smashing this with $T(n)$ yields the sequence
$$ \Sigma^{-1}T(n) \ra T(n) \sm P_{-1} \ra T(n) \sm P_0 \xra{1 \sm tr} T(n).$$

We will see that the key properties of $T(n)$ make it easy to prove the following.

\begin{lem} \label{realizing lem} There exists a map $f(n,r): T(n+r) \ra T(n) \sm P_0 $ such that $f(n,r)_* = q(n,r)$.
\end{lem}

We then let $s(n,r): T(n+r) \ra T(n)$ be the composite
$$ T(n+r) \xra{f(n,r)} T(n) \sm P_0  \xra{1 \sm tr} T(n),$$
and define $X(n,r)$ to be the fiber of $s(n,r)$.

It is easy to then conclude that our short exact sequences have been topologically realized.

\begin{thm} \label{realizing thm}  Applying mod 2 homology to the cofibration sequence
$$\Sigma^{-1}T(n) \ra X(n,r) \ra T(n+r)$$ realizes (\ref{Qnr defn}).
\end{thm}

\begin{addendum} Analogous to Addendum (\ref{homology addendum}), except when $(n,r) = (0,0)$ or $(1,0)$, $f(n,r)$ can be chosen to have the form 
$$T(n+r) \xra{\tilde f(n,r)} T(n) \sm P_1 \xra{1 \sm i} T(n) \sm P_0,$$
where $i$ is the standard right inverse of the projection $P_0 \ra P_1$. (The composite $tr \circ i: P_1 \ra S^0$ is the Kahn-Priddy map.)
\end{addendum}

If $T$ and $X$ are spectra, let $E_k^{s,t}(T,X)$ denote the $k$th page of the mod 2 Adams spectral sequence converging to $[T,X]_{t-s}$.   The maps
$$s(n,r): T(n+r) \ra T(n),$$
which are zero in mod 2 homology, induce maps of spectral sequences \cite{bruner} of the form
$$ s(n,r)^*: \{E_k^{s,t}(T(n),X)\} \ra \{E_k^{s+1,t+1}(T(n+r),X)\}$$
which, by the last theorem, agree with our Dyer--Lashof operations when $k=2$.  The next result follows.

\begin{cor} \label{ASS cor} $ \{E_k^{s,t}(T(\star),X) \} \Rightarrow [T(\star),X]_{t-s}$ is a spectral sequence of modules over the Dyer--Lashof algebra.
\end{cor}

\begin{rem}  The map $f(n,r)$ is not unique: one can add any map of positive Adams filtration to any choice of this without changing its effect on mod 2 cohomology.  We have not studied how this effects either $s(n,r)$ or the 2-completed homotopy type of the finite spectra $X(n,r)$.
\end{rem}

\subsection{Organization of the rest of the paper}

In \secref{background section}, we review the results in the literature that imply \propref{starting point prop}.  In particular, algebraic results in \cite{kuhn mccarty} show this, when combined with an old observation of P.Goerss \cite{goerss unstable ext}.   The constructions and ideas here have their origins in work of W.Singer \cite{singer}, as further developed by H.Miller \cite{miller} and  J.Lannes and S.Zarati \cite{lz}, among others.  

With ingredients which appear in this older work, we  prove \propref{unstable M prop} and \thmref{main thm} in \secref{main thm proof sec}.

We prove \propref{Q(n,n-1) prop} and \thmref{squaring thm} in \secref{squaring sec}.  The formulation of \thmref{squaring thm} was inspired by pondering calculations in \cite{Brian Thomas thesis}, and is new in this revised version.

Background on the geometric properties of the spectra $T(n)$ will be reviewed in \secref{geometric sec}, making it easy to verify \lemref{realizing lem} and \thmref{realizing thm}.

In the final section \secref{applications sec}, we discuss applications, both realized and ongoing.  

This paper had its origins in the author's support of the thesis work of Brian Thomas \cite{Brian Thomas thesis}, which used \thmref{main thm} to understand our Dyer-Lashof operations on $\Ext_A^{*,*}(G(\star), \Sigma^2H_*(ku)) = \Ext_{E(1)}^{*,*}(G(\star), \Sigma^2\Z/2)$.  His project has been extended in \cite{RTG paper}.  

We thank the referee of a first version of this paper for suggesting \propref{unstable M prop}, and implicitly encouraging us to stick to homology and right $A$-module actions.

\section{Algebraic background} \label{background section}

\subsection{Some categories of right $A$--modules} \label{categories sec}

We introduce various categories featured in \cite{kuhn mccarty}.

\begin{itemize}
\item $\M$ is the category of locally finite right $\A$--modules.  The Steenrod squares go down in degree: given $x \in M \in \M$, $|x Sq^i| = |x|-i$.  A right $\A$--module $M$ is locally finite if, for all $x \in M$, $x \cdot \A$ is finite dimensional over $\Z/2$.
\item
$\U$ is the full subcategory of $\M$ consisting of locally finite right $\A$--modules satisfying the unstable condition: $xSq^i = 0$ whenever $2i>|x|$.
\item $\Qq$ is the category of graded vector spaces $M$ acted on by Dyer--Lashof operations $Q^r: M_n \ra M_{n+r}$, for $r \in \Z$, satisfying the Dyer-Lashof Adem relations and the unstable condition: $Q^rx = 0$ whenever $r<|x|$.
\item $\QM$ is the full subcategory of $\M \cap \Qq$ consisting of objects whose Dyer--Lashof operations are intertwined with the Steenrod operations via the Nishida relations.
\item $\QU = \QM \cap \U$.
\end{itemize}

As examples, $H_*(X;\Z/2) \in \M$ if $X$ is a spectrum, $H_*(Z;\Z/2) \in \U$ if $Z$ is a space, $H_*(X;\Z/2) \in \QM$ if $X$ is an $E_{\infty}$-ring spectrum, and $H_*(Z;\Z/2) \in \QU$ if $Z$ is an $E_{\infty}$-space, e.g. if $Z = \Oinfty X$. 

The abelian category $\M$ has enough injectives; indeed, any $M \in \M$ embeds in the injective $M \otimes A_*$, where $A_* = H_*(H\Z/2)$ is viewed as a right $A$-module.  

The following is presumably well known, though we don't know of a reference.

\begin{lem} The abelian category $\M$ is equivalent to the abelian category of comodules over the coalgebra $A_*$.
\end{lem}
\begin{proof}[Sketch Proof] Let $\{\xi^I\}$ and $\{Sq^I\}$ be the usual dual homogeneous bases of $A_*$ and $A$.  If $M$ be an $A_*$-comodule with structure map $\Psi: M \ra M \otimes A_*$, then $M$ is also a locally finite right $A$-module as follows: given $x \in M$, $xSq^J = x_J$ if  $\Psi(x) = \sum_I x_I \otimes \xi^I$.  The finiteness of the sum implies that the right $A$--module structure is locally finite.
\end{proof} 

As in the introduction, we write $\Ext_A^{*,*}(M,N)$ for $\Ext_{\M}^{*,*}(M,N)$. Thanks to the lemma, these correspond to the $\Ext$-groups appearing in the most general presentations of the Adams Spectral Sequence.  To compare with some other versions, e.g. the presentation in \cite[Chap.2]{ravenel green book}, one has the following observation.

\begin{lem}  Let $\ModA$ be the category of all right $A$-modules.  Given $M,N \in \M$, the natural map
$$ \Ext^{*,*}_{A}(M,N) \ra \Ext_{\ModA}^{*,*}(M,N)$$
will be an isomorphism if $N$ is bounded below and of finite type.
\end{lem}
\begin{proof} If $N$ is bounded below and of finite type, then $N \otimes A_*$ will also be injective when viewed in the category of all right $A$--modules, and the cokernel of the inclusion $N \hra N \otimes A_*$ will again be bounded below and of finite type.  From this one deduces that there exists an injective resolution of $N$ in $\M$ that is also an injective resolution when viewed $\ModA$. The lemma follows.
\end{proof}

\subsection{The modules $G(n)$}

Though one can check that $\M$ doesn't have any projectives, the subcategory $\U$ certainly does.  As in the introduction, $G(n)$ is a projective object in $\U$ satisfying 
\begin{equation} \Hom_{\U}(G(n), M) \simeq M_n,
\end{equation}
for all $M \in \U$.  We let $\iota_n \in G(n)_n$ be the universal `top' class.

The modules $G(n)$ can be described in terms of the Steenrod algebra as follows.

\begin{lem} \label{A and G(n) lem} The is an epimorphism of vector spaces $A^{n-k} \era G(n)_k$ with kernel spanned by the Milnor basis elements $Sq^I$ in $A^{n-k}$ of excess more than $n$.  Furthermore, for all $a \in A$, the diagram
For all $a \in A$, we have commutative diagrams
$$
\SelectTips{cm}{}
\xymatrix{
A^{n-k} \ar[d]^{\cdot a} \ar@{->>}[r] & G(n)_k \ar[d]^{\cdot a}  \\
A^{n-k+|a|} \ar@{->>}[r] & G(n)_{k-|a|}. }
$$
\end{lem}
For this description see \cite[Cor.6.14]{genrepI} (with interpretation as in \cite[\S 8]{genrepIII}), or alternatively, dualize the discussion in \cite[\S 2.5]{schwartz book}.

It follows that $G(n)$ has a canonical homogeneous basis corresponding to Milnor basis elements of excess at most $n$.

\subsection{The right derived functors of $\Oinfty$} \label{unstable subsec}

Let $\Oinfty: \M \ra \U$ be the right adjoint to the inclusion $\U \hra \M$.

It is easy to see that the natural map $\Oinfty M \ra M$ can be viewed as the inclusion of the maximal unstable submodule of $M$, and this inclusion induces a natural isomorphism in $\U$
\begin{equation} \label{G(n) hom iso}  (\Oinfty M)_{\star}=\Hom_{\U}(G(\star),\Oinfty M) \simeq \Hom_{A}(G(\star),M)
\end{equation}
for all $M \in \M$.

Now let $\Oinfty_s: \M \ra \U$ be the $s$th right derived functor of $\Oinfty$.  The following lemma follows. 

\begin{lem}\cite[Cor.1.9]{goerss unstable ext} \label{ext lem} The isomorphism (\ref{G(n) hom iso}) induces  natural isomorphisms in $\U$
$$ (\Oinfty_s M)_{\star} \simeq  \Ext^s_{A}(G(\star),M) $$
for all $M \in \M$ and $s \geq 0$.
\end{lem}

\subsection{The free functor $\mathcal R_*$}

\begin{defn} \label{R def}
Let $\mathcal R_*: \M \ra \QM$ be left adjoint to the forgetful functor.  Explicitly,
$\mathcal R_* M = \bigoplus_{s=0}^{\infty} \mathcal R_s M$ where $\mathcal R_s: \M \ra \M$ is given by
$$\mathcal R_sM = \langle Q^Ix \ | \ l(I) = s,  x \in M\rangle/(\text{unstable and Adem relations}).$$
Here, if $I = (i_1,\dots,i_s)$, $Q^Ix = Q^{i_1}\cdots Q^{i_s}x$, and $l(I) = s$.

The right $A$-module structure on $\mathcal R_*M$ is determined by the right $A$-module structure on $M$ toegether with the Nishida relations.
\end{defn}

There are natural transformations $\mu: \mathcal R_s \mathcal R_t M \ra \mathcal R_{s+t}M$, and $\eta: M = \mathcal R_0M\hra \mathcal R_* M$, making $\mathcal R_*: \M \ra \M$ into a monad on $\M$, and $\QM$ is precisely the category of $\mathcal R_*$--algebras.  

Given $M \in \QM$, the structure map $\mathcal R_1 M \ra M$ corresponds to Dyer--Lashof operations $Q^r: M_n \ra M_{n+r}$ in the evident way.

By inspection, one sees the following.

\begin{lem}  $\mathcal R_*$ is exact.
\end{lem}

\begin{defn}  Define $\epsilon: \Sigma \mathcal R_*(M) \ra \mathcal R_*(\Sigma M)$ by $\epsilon(\sigma Q^Ix) = Q^I \sigma x$. This is a natural transformation between functors with values in $\QM$. 
\end{defn}

\subsection{Dyer-Lashof operations on $\Oinfty_s$ and the proof of \propref{starting point prop}}

Now we recall key algebraic results from \cite{kuhn mccarty}.

The next theorem is a restatement of \cite[Thm.1.16]{kuhn mccarty}. As is said in that paper, it is a variant of theorems in \cite{goerss unstable ext}, \cite{lz}, and \cite{powell}, all inspired by \cite{singer}.

\begin{thm} \label{HRM thm} {\bf (a)} \ The formula
$$ d_s(Q^I\sigma^{-1}(x)) = \sum_{i\geq 0} Q^IQ^{i-1}(xSq^i)$$
induces a well defined natural transformation 
$$ d_s: \mathcal R_s(\Sigma^{-1}M) \ra \mathcal R_{s+1}(M)$$
such that $d_*: \mathcal R_*(\Sigma^{-1}M) \ra \mathcal R_{*+1}(M)$ is a map in $\QM$.

\noindent{\bf (b)} \ The composite $\mathcal R_{s-1}(\Sigma^{-2}M) \xra{d_{s-1}} \mathcal R_s(\Sigma^{-1}M) \xra{d_s} \mathcal R_{s+1}(M)$ is zero, and the homology in the middle is naturally isomorphic to $\Sigma^{-1}\Oinfty_s(\Sigma^{-s}M)$.
\end{thm}

\begin{rem}  The isomorphism of this theorem is proved in a roundabout way.  One lets  $R_{*}(M)$ be the cochain complex in $\M$ defined by letting  $R_s(M) = \Sigma \mathcal R_s(\Sigma^{s-1}M)$.  One then shows that $H_s(R_{*}(M)) \simeq \Oinfty_{s}M$ by showing that $R_{*}$ is exact (obvious), that $H_0(R_*(M)) \simeq \Oinfty M$ (easy to check), and finally that $H_s(R_*(\Sigma^n A_*))=0$ for all $s>0$ and all $n \in \Z$.
\end{rem}

\begin{cor} {\bf (a)} If we let $F_s(M) = \Sigma^{-1}\Oinfty_s(\Sigma^{-s}M)$, then $F_*(M)$ is a subquotient of $\mathcal R_*(\Sigma^{-1}M)$ as functors from $\M$ to $\QM$.

{\bf (b)} A short exact sequence $0 \ra M_1 \xra{i} M_2 \xra{j} M_3 \ra 0$ in $\M$ induces an exact triangle in $\QM$: 
\begin{equation*}
\SelectTips{cm}{}
\xymatrix{
 F_*(M_1)  \ar[rr]^{i_*} && F_*(M_2) \ar[dl]^{j_*}  \\
& F_*(M_3), \ar[ul]^{\delta} &  }
\end{equation*}
with $\delta$ raising $*$-degree by 1. Thus $i_*$, $j_*$, and $\delta$ all commute with Dyer-Lashof operations.

\end{cor}

\begin{proof}[Proof of \propref{starting point prop}] This is just statement (a) of the corollary, using that $ \Ext_A^{*,*}(G(\star),M)= \Oinfty_*(\Sigma^{-*}M)_{\star} = F_*(M)_{\star -1}$. 
\end{proof}

Translated into a statement about our Dyer Lashof operations
$$ Q^r: \Ext_A^{s,s}(G(n),M) \ra \Ext_A^{s+1,s+1}(G(n+r),M),$$
statement (b) of the corollary tells us that these commute with dimension shifting, i.e. splicing with short exact sequences. Yoneda's lemma then tell us that these operations all arise from splicing with 
the elements 
$$Q(n,r) = Q^r(1_{G(n)}) \in \Ext_{A}^{1,1}(G(n+r),G(n)),$$
which we regard as short exact sequences
$$ 0 \ra \Sigma^{-1} G(n) \ra Q(n,r) \ra G(n+r) \ra 0,$$
as in \secref{introduction}.
 
We need a couple other properties of the differentials $d_s$.

\begin{lem}\cite[Lemma 4.20]{kuhn mccarty} For all $M \in \M$, the diagram
\begin{equation*}
\SelectTips{cm}{}
\xymatrix{
\Sigma \mathcal R_s(\Sigma^{-1} M) \ar[d]^{\epsilon} \ar[r]^-{\Sigma d_s} & \Sigma \mathcal R_{s+1}(M) \ar[d]^{\epsilon}  \\
\mathcal R_s(M) \ar[r]^-{d_s} & \mathcal R_{s+1}(\Sigma M) }
\end{equation*}
commutes.
\end{lem}

\begin{cor} \label{Qr commutes with suspension cor} The natural map $\epsilon: \Sigma F_*(M) \ra F_*(\Sigma M)$ is a map in $\QU$. Thus Dyer Lashof operations commute with $\epsilon$.
\end{cor} 
Translated into a statement about $\Ext$--groups, this last statement means that the diagram
\begin{equation} \label{epsilon and Ext diagram}
\SelectTips{cm}{}
\xymatrix{
 \Ext_A^{s,s}(G(n),M) \ar[r]^-{Q^r} \ar@{=}[d] &  \Ext_A^{s+1,s+1}(G(n+r),M) \ar@{=}[d] \\
 \Ext_A^{s,s}(\Sigma G(n), \Sigma M) \ar[d]^{p^*_{n}} & \Ext_A^{s+1,s+1}(\Sigma G(n+r),\Sigma M) \ar[d]^{p^*_{(n+r)}} \\
 \Ext_A^{s,s}(G(n+1),\Sigma M) \ar[r]^-{Q^r} & \Ext_A^{s+1,s+1}(G(n+r+1),\Sigma M), }
\end{equation}
commutes, for all $M \in \M$. 

\begin{lem} \cite[Lemma 4.33]{kuhn mccarty} If $M \in \U$, then $d_s: \mathcal R_s(\Sigma^{-1}M) \ra \mathcal R_{s+1}(M)$ is zero.
\end{lem}

Specialized to $s=1$, one learns the following.

\begin{cor}  \label{G(n) res cor} For all $M \in \U$, there is an exact sequence 
$$ 0 \ra \Oinfty(\Sigma^{-1}M) \ra \Sigma^{-1}M \xra{d_0} \Sigma \mathcal R_1(\Sigma^{-1}M) \xra{\rho} \Oinfty_1(\Sigma^{-1}M) \ra 0.$$
\end{cor}

It follows that $\Oinfty_1(\Sigma^{-1}M) \simeq \Ext_A^{1,1}(G(\star),M)$ is spanned by elements of the form $\rho(\sigma Q^r(\sigma^{-1}x))$, $x \in M$. In particular, 
$$Q(n,r)\in \Ext_A^{1,1}(G(n+r),G(n))$$
 corresponds to  
$$\rho(\sigma Q^r(\sigma^{-1}\iota_n))\in \Oinfty_1(\Sigma^{-1}G(n)).$$

\section{Proofs of \propref{unstable M prop} and \thmref{main thm}} \label{main thm proof sec}

If $M$ is unstable, \corref{G(n) res cor} tells us how to relate $\Ext^{s,s}_A(G(\star),M) \simeq \Oinfty_1(\Sigma^{-1}M)$ to the functor $\Sigma \mathcal R_1(\Sigma^{-1}M)$. 

Our goal now is to relate this to (\ref{basic ses}), which was the short exact sequence
$$ 0 \ra \Sigma^{-1}\Z/2 \ra H_*(P_{-1}) \ra H_*(P_0) \ra 0.$$

\begin{rem} We remind the reader that $H_*(P_{-1})$ has a basis given by elements $t_k \in H_k(P_{-1})$ for $k\geq -1$, with $A$-module structure $t_{k+i}Sq^i = \binom{k}{i} t_k$.  In particular, $t_k Sq^{k+1} = t_{-1}$ for all $k$.
\end{rem}

Tensoring (\ref{basic ses}) with $M \in \M$ yields the short exact sequence 
$$ 0 \ra \Sigma^{-1}M \ra M \otimes H_*(P_{-1}) \ra M \otimes H_*(P_0) \ra 0. $$
This, in turn, induces a long exact sequence that begins
\begin{multline*}
0 \ra \Oinfty(\Sigma^{-1}M) \ra \Oinfty(M \otimes H_*(P_{-1})) \ra \Oinfty(M \otimes H_*(P_0)) \\
\xra{\delta_M} \Oinfty_1(\Sigma^{-1}M) \ra \Oinfty_1(M \otimes H_*(P_{-1})) \ra  \dots
\end{multline*}

Note that if $M$ is unstable, then $\Oinfty(M \otimes H_*(P_0)) = M \otimes H_*(P_0)$.  

\begin{thm} \label{unstable M thm} If $M$ is unstable, $\delta_M: M \otimes H_*(P_0) \ra \Oinfty_1(\Sigma^{-1}M)$ is onto.  Explicitly, given $x \in M$, $ \delta_M(q(x,r)) = \rho(\sigma Q^r(\sigma^{-1}x))$, where $q(x,r) = \sum_j x\chi(Sq^j) \otimes t_{r+j}.$
\end{thm}
(Here $\chi$ is the antipode of the Steenrod algebra.)

Thanks to \lemref{ext lem}, the first statement in this theorem is equivalent to the first statement in \propref{unstable M prop}.  

Similarly, the second statement in \propref{unstable M prop} is implied by the next lemma, which admits a short proof.
\begin{lem} If $M\in \M$ has top nonzero degree $n$, then $\Oinfty(M \otimes H_*(P_{-1}))_m = 0$ for $m \geq 2n-1$.
\end{lem}
\begin{proof}  Filtering $M$ by degree, it suffices to prove the lemma when $M = \Sigma^k\Z/2$ for all $k \leq n$.  Using that $t_i Sq^{i+1} = t_{-1}$ for all $i$, It is then easy to check that $\Oinfty(\Sigma^k H_*(P_{-1}))_m = 0$ for all $m \geq 2k-1$.
\end{proof}

Finally, the explicit formula in \thmref{unstable M thm} implies \thmref{main thm}, thanks to the following lemma.

\begin{lem} $\iota_n\chi(Sq^j) \in G(n)_{n-j}$ is the sum of the basis elements in $G(n))_{n-j}$. Thus $q(\iota_n,r) \in G(n) \otimes H_*(P_0)$ is the sum of all the basis elements in degree $(n+r)$.
\end{lem}

\begin{proof} By \lemref{A and G(n) lem}, we have a commutative diagram
$$
\SelectTips{cm}{}
\xymatrix{
A^{0} \ar[d]^{\cdot \chi(Sq^j)} \ar@{=}[r] & G(n)_n \ar[d]^{\cdot \chi(Sq^j)}  \\
A^{j} \ar@{->>}[r] & G(n)_{n-j}, }
$$
and the lemma follows from Milnor's result \cite[Cor. 6]{milnor} that $\chi(Sq^j)$ is the sum of all the (now called) Milnor basis elements in $A^{j}$.
\end{proof}

We now begin the proof of \thmref{unstable M thm}.

Let $M$ be unstable, and consider the following commutative diagram
\begin{equation*}
\SelectTips{cm}{}
\xymatrix{
0 \ar[r]& \Sigma^{-1}M \ar[d]^{d_0} \ar[r] & M \otimes H_*(P_{-1}) \ar[d]^{d_0} \ar[r] & M \otimes H_*(P_0) \ar[d]^{d_0 = 0} \ar[r] & 0 \\
0 \ar[r]& \Sigma\mathcal R_1(\Sigma^{-1}M)  \ar[r] & \Sigma\mathcal R_1(M \otimes H_*(P_{-1}))  \ar[r] & \Sigma\mathcal R_1(M \otimes H_*(P_0))  \ar[r] & 0.}
\end{equation*}

This has exact rows, and we make some observations.  Firstly, since $M$ is unstable, the cokernel of the left vertical map is $\Oinfty_1(\Sigma^{-1}M)$, and the right vertical map is zero because $M \otimes H_*(P_0)$ is unstable.  Thus the middle vertical map lifts to a map
$$\tilde d_0: M \otimes H_*(P_{-1}) \ra \Sigma\mathcal R_1(\Sigma^{-1}M).$$  Factoring out the image of $\Sigma^{-1}M$ from both the domain and range of $\tilde d_0$, induces precisely our connecting map
$$ \delta_M: M \otimes H_*(P_0) \ra \Oinfty_1(\Sigma^{-1}M).$$

We have checked the following lemma.

\begin{lem} \label{ext lemma}  If $M$ is an unstable right module, there is a map $\tilde d_0$ making the diagram
\begin{equation*}
\SelectTips{cm}{}
\xymatrix{
 & M \otimes H_*(P_{-1}) \ar[dl]_{d_0} \ar[d]^{\tilde d_0} \ar[r] & M \otimes H_*(P_0) \ar[d]^{\delta_M} \\
\Sigma \mathcal R_1(M \otimes H_*(P_{-1})) & \Sigma \mathcal R_1(\Sigma^{-1}M) \ar@{>->}[l] \ar@{->>}[r]^{\rho} & \Oinfty_1(\Sigma^{-1}M) }
\end{equation*}
commute.
\end{lem}

The next proposition then finishes the proof of \thmref{unstable M thm}.

\begin{prop} \label{calculating prop} Let $x$ be an element in an unstable module $M$. 

{\bf (a)} \ $\displaystyle \tilde d_0(x \otimes t_r) = \sum_i \sigma Q^{r+i}(\sigma^{-1}xSq^{i})$. \\

\noindent{\bf (b)} \ Let $\displaystyle q(x,r) = \sum_j x\chi(Sq^j)\otimes t_{r+j}$. Then
$\displaystyle \tilde d_0(q(x,r)) = \sigma Q^r(\sigma^{-1}x)$.
\end{prop}

\begin{proof}
The inclusion $\Sigma^{-1}M \hra M \otimes H_*(P_{-1})$ sends $\sigma^{-1}y$ to $y \otimes t_{-1}$.  Thus the commutative triangle in \lemref{ext lemma} tells us two things: firstly, we can calculate $\tilde d_0(x \otimes t_r)$ by instead calculating $d_0(x \otimes t_r)$, and secondly, $d_0(x \otimes t_r)$ must be a linear combination of terms of the form $\sigma Q^{r+i}(y \otimes t_{-1})$ with $y \in M_{|x|-i}$.

We use the formula for $d_0$ in \thmref{HRM thm} to compute:
\begin{equation*}
\begin{split}
d_0(x \otimes t_r) &
= \sum_{i}  \sigma Q^{r+i}((x \otimes t_r)Sq^{r+i+1}) \\
  & = \sum_{i} \sum_j \sigma Q^{r+i}(x Sq^{i+j}\otimes t_rSq^{r+1-j}) \\
    & = \sum_{i} \sigma Q^{r+i}(x Sq^i\otimes t_{-1}). \\
\end{split}
\end{equation*}
The last equality here holds because, firstly, $t_rSq^{r+1} = t_{-1}$ for all $r$, and, secondly,  all the terms in the double sum with $j>0$ must be zero by the second observation above.  Thus we see that statement (a) of the proposition holds.

Statement (b) now follows from a straightforward calculation:
\begin{equation*}
\begin{split}
\tilde d_0(q(x,r)) & = \sum_j \tilde d_0(x \chi(Sq^j)\otimes t_{r+j}) \\
& = \sum_j \sum_ i \sigma Q^{r+j+i}( \sigma^{-1}x\chi(Sq^j) Sq^{i}) \\
& = \sum_k \sigma Q^{r+k}(\sigma^{-1} (\sum_{i+j=k} x\chi(Sq^j) Sq^{i})) \\
& = \sigma Q^r \sigma^{-1} x,
\end{split}
\end{equation*}
since
$
\displaystyle \sum_{i+j=k} x\chi(Sq^j) Sq^{i} =
\begin{cases}
x & \text{if } k=0 \\ 0 & \text{otherwise. }
\end{cases}
$
\end{proof}

\begin{rem}  Maps like $\tilde d_0$ appear in the literature related to the Singer construction.  See, for example, \cite[Lemma 1.2]{miller} (written in cohomology).  
This $A$-module map also has a geometric origin. If $D_2(X) = X^{\sm 2}_{h\Sigma_2}$, then $\mathcal R_1(H_*(X)) \subseteq H_*(D_2(X))$. Note that $P_{-1} = \Sigma D_2(S^{-1})$.  If $X$ is a space, there is a natural map
$\Delta_X: X \sm D_2(S^{-1}) \ra  D_2(\Sigma^{-1}X)$, and 
$\Sigma \Delta_X$ induces $\tilde d_0: H_*(X) \otimes H_*(P_{-1}) \ra \Sigma \mathcal R_1(\Sigma^{-1}H_*(X))$ in homology.
\end{rem}
\section{Proofs of \propref{Q(n,n-1) prop} and \thmref{squaring thm}} \label{squaring sec}

In this section, we first check the identification of $Q(n,n-1)$ and $Q(n,n)$ as in \propref{Q(n,n-1) prop}.

We start with a useful characterization of the Mahowald short exact sequence.

\begin{lem} The Mahowald sequence 
$$0 \ra G(n) \xra{Sq^n \cdot} G(2n) \xra{p_{2n-1}} \Sigma G(2n-1) \ra 0$$ is the unique nonsplit extension in $\Ext_A^1(\Sigma G(2n-1), G(n))$ that splits after pulling back by $p_{2n-1}$.
\end{lem} 
\begin{proof} The Mahowald sequence {\em is} nonsplit and induces
$$ \Hom_A(G(n),G(n)) \xra{\delta} \Ext^1_A(\Sigma G(2n-1),G(n)) \xra{p_{2n-1}^*} \Ext^1_A(G(2n),G(n)),$$
exact in the middle.  By construction, the Mahowald sequence corresponds to $\delta(1_{G(n)})$, the unique nonzero element in $\ker(p_{2n-1}^*)$. 
\end{proof}

Now we give a diagrammatic consequence of \corref{Qr commutes with suspension cor}.

\begin{lem} \label{pushout pullback lemma}  There is a commutative diagram of $A$--modules
\begin{equation*}
\SelectTips{cm}{}
\xymatrix{
0  \ar[r] &  G(n)   \ar[r]& \Sigma Q(n,r)   \ar[r]&  \Sigma G(n+r)  \ar[r] & 0 \\
0  \ar[r] &  G(n) \ar@{=}[u]  \ar[r]& M \ar[u]  \ar[r]&  G(n+r+1) \ar[u]_{p_{n+r}} \ar[r] & 0 \\
0  \ar[r] & \Sigma^{-1} G(n+1)  \ar[u]_{p_n} \ar[r]& Q(n+1,r) \ar[u] \ar[r]&   G(n+r+1) \ar@{=}[u] \ar[r] & 0, }
\end{equation*}
in which the upper right square is a pullback, and the lower left square is a pushout, and the rows are exact.
\end{lem}

\begin{proof} Specialization of (\ref{epsilon and Ext diagram}) to the case when $s=0$ and $M=G(n)$ tells us that there is a commutative diagram
\begin{equation*} 
\SelectTips{cm}{}
\xymatrix{
 \Hom_A(G(n),G(n)) \ar[r]^-{Q^r} \ar@{=}[d] &  \Ext_A^{1,1}(G(n+r),G(n)) \ar@{=}[d] \\
 \Hom_A(\Sigma G(n), \Sigma G(n)) \ar[d]^{p^*_{n}} & \Ext_A^{1,1}(\Sigma G(n+r),\Sigma G(n)) \ar[d]^{p^*_{(n+r)}} \\
 \Hom_A(G(n+1),\Sigma G(n)) \ar[r]^-{Q^r} & \Ext_A^{1,1}(G(n+r+1),\Sigma G(n)). }
\end{equation*}
The image of $1_{G(n)}$ under the upper horizontal map followed by the right vertical map is the middle sequence of the diagram of the lemma viewed as the pullback of the top sequence by the map $p_{n+r}$.  Meanwhile, the image of $1_{G(n)}$ under the left vertical map followed by the lower horizontal map is $Q^r(p_n) = Q^r(p_{n*}( 1_{G(n+1)})) = p_{n*}(Q^r(1_{G(n+1)}))$, which is the middle sequence of the  diagram of the lemma viewed as the pushout of the bottom sequence by the map $p_n$.   
\end{proof}

\begin{proof}[Proof of \propref{Q(n,n-1) prop}]

$Q^{n-1}(1_{G(n)})$ corresponds to the nonsplit short exact sequence
\begin{equation} \label{Q(n,n-1) sequence}  0 \ra G(n) \ra \Sigma Q(n,n-1) \ra \Sigma G(2n-1) \ra 0.
\end{equation}
\lemref{pushout pullback lemma} tells us that $p_{2n-1}^*(Q^{n-1}(1_{G(n)})) = p_{n*}(Q^{n-1}(1_{G(n+1)})) = 0$, since $Q^{n-1}(1_{G(n+1)}) = 0$ for degree reasons. Thus (\ref{Q(n,n-1) sequence}) splits after pulling back by $p_{2n-1}$ and so agrees with the Mahowald sequence.  We have proved \propref{Q(n,n-1) prop}(a).

The proof of \propref{Q(n,n-1) prop}(b) now follows easily. \lemref{pushout pullback lemma} tells us that there is a commutative diagram with exact rows
\begin{equation*}
\SelectTips{cm}{}
\xymatrix{
0  \ar[r] &  G(n)   \ar[r]& \Sigma Q(n,n)   \ar[r]&  \Sigma G(2n)  \ar[r] & 0 \\
0  \ar[r] &  G(n) \ar@{=}[u]  \ar[r]& M \ar[u]  \ar[r]&  G(2n+1) \ar[u]_{p_{2n}} \ar[r] & 0 \\
0  \ar[r] & \Sigma^{-1} G(n+1)  \ar[u]_{p_n} \ar[r]& Q(n+1,n) \ar[u] \ar[r]&   G(2n+1) \ar@{=}[u] \ar[r] & 0, }
\end{equation*}
in which the upper right square is a pullback and the lower left square is a pushout. Now note that $p_{2n}$ is an isomorphism, and that the bottom row has just been identified as being equivalent to the sequence
$$ 0 \ra \Sigma^{-1} G(n+1)  \xra{ Sq^{n+1} \cdot}  \Sigma^{-1} G(2n+2) \xra{p_{2n+1}} G(2n+1) \ra 0.$$

\end{proof}

Now we turn to the proof of \thmref{squaring thm}, which said that for all right $A$--modules $M$, the diagram 
\begin{equation} \label{squaring diagram}
\SelectTips{cm}{}
\xymatrix{
\Ext_A^{s,s}(G(2n),M) \ar[d]^{Sq^n} \ar[rrr]^-{h_0 } &&& \Ext_A^{s+1,s+1}(G(2n),M) \ar[d]^{Sq^n}  \\
\Ext_A^{s,s}(G(n),M) \ar[urrr]^-{Q^n} \ar[rrr]^-{(n+1)h_0} &&& \Ext_A^{s+1,s+1}(G(n),M) }
\end{equation}
commutes.

As usual, to show that the top triangle commutes, we (just) need to check that the triangle 
\begin{equation} \label{hom squaring diagram}
\SelectTips{cm}{}
\xymatrix{
\Hom_A(G(2n),G(2n)) \ar[d]^{Sq^n} \ar[rrr]^-{h_0 } &&& \Ext_A^{1,1}(G(2n),G(2n))   \\
\Hom_A(G(n),G(2n)) \ar[urrr]^-{Q^n} &&&  }
\end{equation}
commutes.

By definition, $Sq^nQ^n(1_{G(2n)})$ is represented by the bottom line of this pushout diagram:
\begin{equation} \label{squaring diagram 2}
\SelectTips{cm}{}
\xymatrix{
0 \ar[r] & \Sigma G(n)  \ar[d]^{Sq^n} \ar[r] & \Sigma^2 Q(n,n) \ar[d] \ar[r] & \Sigma^2 G(2n) \ar@{=}[d] \ar[r] &  0  \\
0 \ar[r] & \Sigma G(2n) \ar[r] & P \ar[r] & \Sigma^2 G(2n) \ar[r] &  0. }
\end{equation}
Diagram (\ref{hom squaring diagram}) commutes if this bottom line is equivalent to 
$$ 0 \ra \Sigma G(2n) \ra G(2) \otimes G(2n) \ra \Sigma^2 G(2n) \ra 0.$$

Now \propref{Q(n,n-1) prop}(b) tells us that the top line of diagram (\ref{squaring diagram 2}) is itself the bottom line of this pushout diagram:
\begin{equation} \label{squaring diagram 3}
\SelectTips{cm}{}
\xymatrix{
0 \ar[r] &  G(1+n)  \ar[d]^{p_n} \ar[r]^{Sq^{1+n}} & G(2+2n) \ar[d] \ar[r] & \Sigma G(2n+1) \ar[d]^{\wr} \ar[r] &  0  \\
0 \ar[r] & \Sigma G(n)   \ar[r] & \Sigma^2 Q(n,n)  \ar[r] & \Sigma^2 G(2n)  \ar[r] &  0.  }
\end{equation}
Thus the next lemma will finish the proof that (\ref{hom squaring diagram}) commutes.
\begin{lem} The diagram 
\begin{equation} \label{squaring diagram 4}
\SelectTips{cm}{}
\xymatrix{
0 \ar[r] &  G(1+n)  \ar[d]^{p_n} \ar[r]^{Sq^{1+n}} & G(2+2n) \ar[dd]^p \ar[r] & \Sigma G(2n+1) \ar[dd]^{\wr} \ar[r] &  0  \\
& \Sigma G(n) \ar[d]^{Sq^n}   &    &  &   \\  
0 \ar[r] & \Sigma G(2n) \ar[r] & G(2) \otimes G(2n) \ar[r] & \Sigma^2 G(2n) \ar[r] &  0.}
\end{equation}
commutes, where $p$ is the map that is nonzero in degree $(2+2n)$.
\end{lem}
\begin{proof} The right square commutes, as this can be checked in degree $(2+2n)$ and all four maps are nonzero in this degree.  

To check the left square commutes, we need to check that this is the case when evaluated on $\iota_{1+n}$. 

The top left map sends $\iota_{1+n}$ to $i_{2+2n}Sq^{1+n}$. Since $p(\iota_{2+2n}) = \iota_2 \otimes \iota_{2n}$, we can compute:
\begin{equation*}
\begin{split}
p(i_{2+2n}Sq^{1+n}) &
= (\iota_2 \otimes \iota_{2n})Sq^{1+n} \\
  & = \iota_2 \otimes \iota_{2n}Sq^{1+n} + \iota_2Sq^1 \otimes \iota_{2n}Sq^n \text{ (by the Cartan formula)} \\
  & = \iota_2Sq^1 \otimes \iota_{2n}Sq^n \text{ (by the unstable condition)}.
\end{split}
\end{equation*}

This agrees with the composite in the other direction: the left two vertical maps send $\iota_{1+n}$ to $\sigma \iota_{2n}Sq^n$, and then the lower left map sends $\sigma x$ to $\iota_2Sq^1 \otimes x$ for all $x \in G(2n)$.
\end{proof}

To finish the proof of \thmref{squaring thm}, we use what we have just proved to show that the lower triangle in (\ref{squaring diagram}) commutes.  The Dyer-Lashof Adem relations tell us that 
$$ (Q^nx)Sq^n = \sum_i \binom{0}{n-2i} Q^i(xSq^i) = \begin{cases}
Q^m(xSq^m) = h_0x & \text{if } n = 2m \\ 0 & \text{if $n$ is odd}.
\end{cases}
$$
This rewrites as $(Q^nx)Sq^n = (n+1)h_0x$.

\section{Geometric realization} \label{geometric sec}

As in \cite{kuhn83}, call a spectrum $X$ {\em spacelike} if it is a retract of a suspension spectrum.

The finite spectrum $T(n)$ is the $n$-dual of the $n$th Brown-Gitler spectrum\footnote{This is logically backwards. In \cite{bg} Brown and Gitler first construct $T(n)$ and then study the homology theory defined by its $n$-dual. The notation $T(n)$ was used by them, and in subsequent papers by Lannes, Goerss, this author, and others, and should not be confused with telescopic $T(n)$, used by many, including this author.}, and has the following remarkable properties:
\begin{enumerate}
\item $H_*(T(n)) \simeq G(n)$.
\item (Brown-Gitler property) If $X$ is any spacelike spectrum, the natural map
$[T(n),X] \ra H_n(X)$, sending $f$ to $f_*(\iota_n)$, is onto.
\item $T(n)$ is spacelike.
\end{enumerate}
Brown and Gitler \cite{bg} showed that there are $T(n)$ satisfying the first and second property, and the third property was proved by Goerss \cite{goerss T(n) is spacelike} and Lannes \cite{lannes T(n) are spacelike}.  A nice proof of all of this is given in \cite{glm}, and the equivalence of the second and third properties, assuming the first, is shown in \cite{hunter kuhn}.

Using these properties, it is easy prove \lemref{realizing lem} and \thmref{realizing thm}.  Since $T(n)$ is spacelike, so is $T(n) \sm P_0$.  The Brown-Gitler property of $T(n+r)$ then implies that there exists $f(n,r): T(n+r) \ra T(n) \sm P_0$ such that $f_*(\iota_{n+r})$ equals the sum of all the basis elements of $G(n) \otimes H_*(P_0)$ in degree $(n+r)$. Thus $f(n,r)_* = q(n,r)$, proving \lemref{realizing lem}.

As in the introduction,  we now let  $s(n,r): T(n+r) \ra T(n)$ be the composite
$$ T(n+r) \xra{f(n,r)} T(n) \sm P_0 \xra{1 \sm tr} T(n).$$
If one lets $X(n,r)$ be the fiber of $s(n,r)$, one gets a commutative diagram
\begin{equation*}
\SelectTips{cm}{}
\xymatrix{\Sigma^{-1} T(n)  \ar[r] \ar@{=}[d] & X(n,r) \ar[r] \ar[d] & T(n+r) \ar[d]^{f(n,r)} \ar[r]^-{s(n,r)} & T(n) \ar@{=}[d]  \\
\Sigma^{-1} T(n)  \ar[r] & T(n) \sm P_{-1} \ar[r] & T(n) \sm P_0 \ar[r]^-{tr} & T(n),   \\
 }
\end{equation*}
in which the horizontal rows are cofibration sequences.

By construction, we see that applying homology to the left two squares of this diagram realizes the diagram of $A$-modules appearing in \thmref{main thm}, and we have proved \thmref{realizing thm}.

\section{Towards applications} \label{applications sec}
Given a spectrum $X$, consider the following two conditions that may or may not hold:
\begin{itemize}

\item[(a)] (geometric condition) \ $[T(\star),X] \ra \Hom_A(G(\star), H_*(X))$ is onto.

\item[(b)] (algebraic condition) \ $\Hom_A(G(\star), H_*(X))$ generates \\ $\bigoplus_{s=0}^{\infty}\Ext_A^{s,s}(G(\star), H_*(X))$ as a module over the Dyer-Lashof algebra.
\end{itemize}

We have an obvious consequence of \corref{ASS cor}.

\begin{cor} If conditions (a) and (b) hold for a spectrum $X$, then, for all $s$ and $n$, $\Ext^{s,s}_A(G(\star), H_*(X))$ consists of Adams spectral sequence permanent cycles.
\end{cor}

Conditions (a) and (b) also come up in a rather different context. The paper \cite{kuhn mccarty} was a study of the 2nd quadrant mod 2 homology spectral sequence $\{E^r_{*,*}(X)\}$ associated to the Goodwillie tower of the functor from spectra to spectra sending $X$ to $\Sinfty_+ \Oinfty X$.   One of the main results is that $E^{\infty}_{*,*}(X)$ is an algebraic functor of the right $A$--module $H_*(X)$ when the following two conditions hold:
\begin{itemize}
\item[(a')] (geometric condition) \ The evaluation map $e: \Sinfty \Oinfty X \ra X$ induces an epimorphism $e_*: H_*(\Oinfty X) \ra \Oinfty H_*(X)$.

\item[(b')] (algebraic condition) \ $\Oinfty H_*(X)$ generates  $\bigoplus_{s=0}^{\infty}L_s(H_*(X))$ as a module over the Dyer-Lashof algebra where, if $M \in \M$,
$$L_s(M) = \IM \{ \Oinfty_s(\Sigma^{-s} M) \xra{\epsilon} \Sigma^{-1}\Oinfty_s(\Sigma^{1-s} M)\}.$$
\end{itemize}

Explicitly, when (a') and (b') hold, \cite[Cor.1.14]{kuhn mccarty} says that one has an algebraic description of the associated graded of $H_*(\Oinfty X)$: there is then an isomorphism of algebras in $\QU$:
\begin{equation} \label{Einfty formula} E^{\infty}_{*,*}(X) \simeq S^*(\bigoplus_{s=0}^{\infty} L_s(H_*(X)))/(x^2 - Q^{|x|}(x)),
\end{equation}
where $x \in L_s(H_*(X))_n$ has bidegree $(-2^s,2^s+n)$ in $E^{\infty}_{*,*}(X)$.

Conditions (a') and (b') were shown to be satisfied in two extreme cases: when $X$ is a suspension spectrum or when $X$ is an Eilenberg--MacLane spectrum, and one wondered about other examples.

To help with this, we relate conditions (a') and (b') to (a) and (b).

\begin{prop} \label{a = a' prop} Condition (a) is equivalent to condition (a'). 
\end{prop}
We postpone the proof.

More clearly,  there is an epimorphism of modules over the Dyer-Lashof algebra
\begin{equation} \label{epi equation} \bigoplus_{s=0}^{\infty} \Ext_A^{s,s}(G(\star), M) = \bigoplus_{s=0}^{\infty} \Oinfty_s(\Sigma^{-s}M) \era \bigoplus_{s=0}^{\infty} L_s(M),
\end{equation}
and so condition (b) implies condition (b').

\begin{rem}  One can show that if $\Ext_{A}^{s,s}(G(n), M) = 0$ whenever $n$ is odd, then epimorphism (\ref{epi equation}) will be an isomorphism.
\end{rem}

Now the idea is to study condition (b) using \thmref{main thm}, especially in situations when condition (a) (or (a')) holds.  Here is one intriguing family of spectra where this happens.

\begin{prop} \label{condition (a) lemma}  Condition (a') holds if $X = \Sigma^2 BP\langle h \rangle$, for $0\leq h \leq \infty$ (so the spectral sequences studied in \cite{kuhn mccarty} are converging to the mod 2 homology of the infinite loopspaces $\mathbb C \mathbb P^{\infty}$, $BU$, $BP\langle 2 \rangle_2$, \dots, $BP_2$, equipped with Dyer--Lashof and Steenrod operations).
\end{prop}
We also postpone the proof of this proposition.

$H^*(BP\langle h \rangle) = A//E(h)$, where $E(h)$ is the sub-Hopf algebra of $A$ which, as an algebra, is the exterior algebra on the Milnor primitives $Q_0, Q_1, \dots, Q_h$.  One is led to the problem of understanding $\bigoplus_s \Ext_{E(h)}^{s,s}(G(\star), \Sigma^2 \Z/2)$ as a module over the Dyer--Lashof algebra.

This is already interesting when $h=1$, and $BP\langle 1 \rangle = ku$.
The thesis of Brian Thomas \cite{Brian Thomas thesis}, and a related NSF RTG collaborative research project \cite{RTG paper}, study this case, by using the description of our $A$-module extensions given in \thmref{main thm} to study these extensions restricted to $E(1)$. 

Thomas uses this to show that when $X = \Sigma^2 ku$, condition (b) holds. His calculations, together with our new \thmref{squaring thm}, let one use (\ref{Einfty formula}) to show that $E^{\infty}_{*,*}(\Sigma^2ku)$ is polynomial on even degree classes, in agreement with $H_*(BU)$.

The RTG collaboration has reworked some of Thomas' results, and determined when $Q(n,r)$ is free over $E(1)$ (and even $A(1)$) with implications for the finite spectra $X(n,r)$.   One easy-to-state conclusion is that, if $2^k>n$, then $X(n,2^k)$ is a finite spectrum of type 2.  This seems a rather novel source of type 2 complexes, and one wonders if there are some systematic ways of building higher type complexes from Brown-Gitler spectra.

\begin{proof}[Proof of \propref{a = a' prop}]

We wish to show that conditions (a) and (a') are equivalent.

Recall that, given $M \in \M$, $\Hom_A(G(n),M) = \Oinfty M_n$, the degree $n$ part of $\Oinfty M$.

Condition (a) thus says that  $[T(n),X] \ra \Oinfty H_n(X)$ is onto for all $n$. 

Recall that $e: \Sinfty \Oinfty X \ra X$ is the evaluation map.  Since $\Sinfty \Oinfty X$ is a suspension spectrum, its homology is unstable, and thus the image of $e_*$ will be a submodule of $\Oinfty H_*(X)$. Condition (a') then says that $e_*: H_n(\Oinfty X) \ra \Oinfty H_n(X)$ is onto for all $n$.

We have a commutative diagram
\begin{equation*}
\SelectTips{cm}{}
\xymatrix{
[T(n),\Sinfty \Oinfty X] \ar[d]^{(iv)} \ar[r]^-{(iii)}_-{e_*} & [T(n),X] \ar[d]^{(ii)}  \\
H_n(\Oinfty X) \ar[r]^-{(i)}_-{e_*} & \Oinfty H_n(X), }
\end{equation*}
where the two vertical homomorphisms send a map $f$ to $f_*(\iota_n)$.  The map (iv) is onto by the Brown-Gitler property of $T(n)$, while a standard argument \cite[Prop. 2.4]{kuhn83} shows that (iii) is onto because $\Oinfty e$ has a section and $T(n)$ is spacelike.  It follows that (i) is onto if and only if (ii) is onto.
\end{proof}

\begin{proof}[Proof of \propref{condition (a) lemma}]

We need to show that
$$e_*: H_*(\Oinfty \Sigma^2 BP\langle h\rangle) \ra \Oinfty H_*(\Sigma^2 BP\langle h\rangle)$$
is onto for all $h$.

A complex orientation $u_h \in BP\langle h\rangle^2(\mathbb C \mathbb P^{\infty})$ can be viewed as a map $u_h: \Sinfty \mathbb C \mathbb P^{\infty} \ra \Sigma^2 BP\langle h\rangle$, and it is formal that this lifts through
$$e: \Sinfty \Oinfty \Sigma^2 BP\langle h\rangle \ra \Sigma^2 BP\langle h\rangle.$$  Thus it suffices to show that
$$ u_{h*}: H_*(\mathbb C \mathbb P^{\infty}) \ra \Oinfty H_*(\Sigma^2 BP\langle h\rangle)$$
is onto.

The natural map $BP\langle h \rangle \ra BP\langle 0 \rangle = H\Z$ induces the epimorphism $A//E(0) \ra A//E(h)$ in cohomology, and thus a monomorphism
$$\Oinfty H_*(\Sigma^2 BP\langle h\rangle) \ra \Oinfty H_*(\Sigma^2 H\Z).$$
Since the composite $\Sinfty \mathbb C \mathbb P^{\infty} \xra{u_h} \Sigma^2 BP\langle h\rangle \ra \Sigma^2H\Z$ is $u_0$, if we show $u_{0*}: H_*(\mathbb C \mathbb P^{\infty}) \ra \Oinfty H_*(\Sigma^2 H\Z)$ is onto, the lemma will follow for all $h$ (and those monomorphisms must be isomorphisms).

The map $u_{0*}: H_*(\mathbb C \mathbb P^{\infty}) \ra \Oinfty H_*(\Sigma^2 H\Z)$ is certainly onto in degree 2, and then Steenrod operations show that it will also be nonzero in degrees $2^k$ for all $k\geq 1$.  Meanwhile the range of $u_{0*}$ can be easily computed:
$$ \Oinfty H_n(\Sigma^2 H\Z) \simeq \Hom_{E(0)}(G(n), \Sigma^2 \Z/2),$$
which is easily checked to be nonzero and one dimensional exactly when $n = 2^k$ for $k\geq 1$.
\end{proof}

\end{document}